\documentclass[10pt]{article}
\usepackage{amsmath}
\usepackage{amssymb}
\usepackage{indentfirst}
\usepackage[latin1]{inputenc}
\usepackage{geometry}
\usepackage{amsfonts}

\newcommand{\Z} {\mathbb{Z}}

\newcommand{\U}{{\cal U}}

\newcommand{\bc}{\begin{center}}
\newcommand{\ec}{\end{center}}

\newcommand{\be}{\begin{enumerate}}
\newcommand{\ee}{\end{enumerate}}
\newcommand{\bd}{\begin{description}}
\newcommand{\ed}{\end{description}}

\newtheorem{bdf}{Definition}[section]
\newtheorem{bth}[bdf]{Theorem}
\newtheorem{bco}[bdf]{Corolarry}
\newtheorem{ble}[bdf]{Lemma}
\newtheorem{bpr}[bdf]{Proposition}

\newcommand{\bob}{\begin{Ob}}
\newcommand{\eob}{\end{Ob}}

\newcommand{\qed}{\enspace\vrule  height6pt  width4pt  depth2pt}
\newenvironment{proof}{\par\noindent{\bf Proof.}}{$\qed$\par\bigskip}

\begin{document}

\title{Traces of Torsion Units}
\author{Juriaans, S. O. \thanks{ Research supported by CNPq-Brazil.
} \and De A. E Silva, A.
\and Souza Filho, A. C.  \thanks{ Research supported by FAPESP-Brazil.}}
\date{}
\maketitle

\begin{abstract}
A conjecture due to  Zassenhaus asserts that if $\ G$ is a finite group then any torsion unit  in $\mathbb{Z}G$ is  conjugate in $\mathbb{Q}G$ to an element of $\ G$.  Here a   weaker form of this conjecture is proved for some infinite groups. 
\end{abstract}

\footnotetext{
AMS Subject Classification: Primary 20C05, 20C07, 16S34. Secondary 16U60.%
\newline
Key words and frases: Group Rings, torsion units, unique trace property.}

\section{Introduction}

Let $G$ be a group and let $\U_1(\Z G)$ be the group of units of
augmentation one of the integral group ring $\Z G$. Given
elements $\alpha = \sum {\alpha} (g)g \in \mathbb{Z} G$ and $g\in G$%
, we denote by $C_g$ the conjugacy class of $g$ in $G$ and set $\tilde{\alpha}(g) =
\sum_{h\in C_g} \alpha (h)$. If $G$ is a finite group a 
conjecture of Zassenhaus (see \cite{Sehgal, Valenti, Weiss}) states that every torsion element  $\alpha\in \U_1(\Z G)$ is  rationally conjugate to a group element. For finite groups  this is equivalent to the following (see \cite{Bovdi, Valenti}):  for every $\gamma\in \langle \alpha\rangle $ there exists an element $g_0 \in G $, unique up to
conjugacy, such that $\tilde{\gamma}(g_0)\neq 0$. This leads to the following definition.  A unit $\alpha \in \U_1(\mathbb{Z} G)$ is said to have the {\it unique trace property} if there exists an element $g\in G$, unique up to
conjugacy, such that $\tilde{\alpha}(g)\neq 0$. As in \cite{Bovdi},  
a group $G$ has the {\it unique trace property} ({\bf UT}-{\it property}) if
every element $\alpha \in \U_1(\mathbb{Z} G) $ of finite order has the
unique trace property. In \cite{Bovdi} it is proved that nilpotent groups are UT-groups. Let $p$ be a rational prime. We say that a group $G$ is a {\bf p-UT} group
if every torsion unit of prime power order has the unique trace property.

The paper is organized as follows. In the next section we prove some preliminary results which are used in the last section   to exhibit  new  classes of UT and $p$-UT groups. The main difficulty is to show that  if $\ G$ is a group and  $\ \alpha\in \mathbb{Z}G$ is a torsion unit, then $\tilde{\alpha}(g)=0$ for elements $g\in G$ of infinite order. This, together with a reduction to the finite case, is the main tool used here.  The first results in this direction are from \cite{Bovdi}.


\section{Preliminary Results}

We begin by proving a result which, in some cases, deals with the traces of elements of infinite order.

\begin{bpr}
\label{prop1}
Let $G$ be a locally noetherian by finite group and  $\alpha \in \U_1(\mathbb{Z} G)$ 
 a torsion unit. Then  $\tilde{\alpha}(g)=0$ for any $g\in G $ of infinite
order.
\end{bpr}

\begin{proof} Suppose that $\tilde{\alpha}(g)\neq 0$. Then, by \cite[
Prop. 2]{Bovdi} there exists an integer $k>1$ and an element $x\in G$ such that $%
x^{-1}gx=g^k$. If $x$ is of finite order, set $m=o(x)$. Then $g=x^{-m}g{x^m}%
=g^{k^m}$ and hence we have a contradiction.

Suppose that  $o(x)=\infty$ and let $H$ be a normal locally noetherian subgroup of finite index in $G$. Then  there exists an integer $m>0$ such that $x^m\ and \ g^m $ are in $H$. Set $t=x^m,\ h=g^m,\
n=k^m$ and $H_0=\langle t,h \rangle $. Then $t^{-1}ht=h^n$, and since $H_0$ is noetherian we must have that $n=1$ and
consequently $k=1$,  a contradiction. \end{proof}

The previous result extends one of  \cite{Bovdi}. Note that to prove the proposition we do not need $H$ to be normal. The only thing we need is that for any $g\in G$ of infinite order, there exists an integer $n=n(g)$ such
that $g^n\in H$. As a consequence we have the following result.

\begin{bco}
Let $G$ be a group, $H$ a normal locally noetherian torsion free subgroup of finite index in $G$  and $A<\U_1(\mathbb{Z} G)$ a
finite subgroup. Then the order of $A$ divides $[G:H]$.
\end{bco}

\begin{proof} If we denote by $\Psi$ the natural projection of $%
\mathbb{Z} G$ onto $\mathbb{Z} (G/H)$ then we just have to
show that $\Psi$ is injective on $A$. In order to show this let $\alpha$ be
an element of $A$ which is mapped to 1 by $\Psi$. We have $1=\Psi(\alpha)=\sum_{g \notin H}\alpha_g \Psi(g)+\sum_{g \in H}\alpha_g \Psi(g)=\sum_{g \notin H}\alpha_g \Psi(g)+\alpha (1)$. In the last equality we used the torsion freeness of   $H$ and Proposition~\ref{prop1}. It follows that  $1 \in supp(\alpha)$ and thus, by \cite[Theorem 7.3.1]{pms}, $\alpha=1$. \end{proof}

The infinite dihedral group has a normal cyclic subgroup of index 2. Hence a non-trivial finite subgroup of $\U_1(\mathbb{Z} G)$ has order 2 (see   \cite{Sehgal}).

Given a group $G$,  $T(G)$ denotes the set of elements of finite order of $G$. In general this is not a subgroup of $G$.


\section{The UT and $p$-UT Property}

In this section we study the UT and $p$-UT-property and  give examples of classes of groups having one of these  properties.

\begin{bth}
Let $G$ be a group, $H$ a normal locally noetherian torsion free subgroup of finite index and suppose that $T(G)$ is a subgroup. If $G/H$ is a UT-group then $G$ is a also a UT-group.
\end{bth}

\begin{proof} Let $\alpha \in \U_1(\mathbb{Z} G)$ be a torsion
unit and $g\in G$ an element. If $g$ is of infinite order then, by  
 Proposition~\ref{prop1}, we have that $\tilde{\alpha}(g)=0$.

If $g$ has finite order, denote by $\beta $ and $\overline{ g }$ the
projections of $\alpha$ and $g$ in $\U_1(\mathbb{Z} (G/H))$. Let $%
C_{\overline g}$ be the conjugacy class of $\overline{ g }$. Then it is easy to see that 
$C_{\overline g}$ is the projection of the subset \newline
$S=\{k\in G:k=t^{-1}gth, \ h\in H\, t\in G\} $. Since $T(G)$ is a normal
subgroup and $H$ is normal and torsion free, we see that $S\cap T(G)=C_g$.
Furthermore, if we write $S=S_{1}\cup C_g$, where $S_{1}$ are the elements
of infinite order of $S$, then $S_{1}$ is a normal subset of $G$. Writing $%
S_1$ as a disjoint union of conjugacy classes and applying 
Proposition~\ref{prop1}, it follows that $\sum_{h\in S_1}\alpha (h)=0$ and hence   $\tilde{\beta}(\overline g)= \sum_{h\in S}\alpha (h)=%
\tilde{\alpha}(g)$. Since, by our assumption, $G/H$ is a UT-group the result
follows. \end{proof}

\begin{bco}  [\cite{Bovdi}]
Let $G$ be a locally nilpotent group. Then $G$ is a UT-group.
\end{bco}

\begin{proof} We may suppose that $G$ is  finitely generated. This, together with   \cite[5.4.6]{Robinson},     \cite[5.4.15]{Robinson} and \cite{Weiss},  gives that the hypothesis of the theorem are satisfied. Hence $G$ is a UT-group.\end{proof}

If $G$ is a group and $g\in G$ is an element we denote by $K(g)=[g,G]$. Let
now $G$ be a group generated by an element $t$ and an abelian normal
subgroup $A$ such that $t^{-1}at=a^{-1}$ for any $a\in A$ and $t^2 \in A$.
Let $\alpha \in \U_1(\mathbb{Z} G)$ be a torsion unit. By Proposition \ref{prop1}, we
have that $\tilde{\alpha}(g)=0 $ for every element of infinite order. Now
let $g=ta\in G$ be an element which is not in $A$. We compute $K(g)$. If $%
b\in A$ then $[g,b]=[t,b]=b^{-2}$. If $h=tb$ then $%
[g,h]=[ta,tb]=[tb,t][a,tb]=[b,t][a,t]=(ba)^{-2}$. Hence $K(g)=\{a^2 : a\in
A\}$. So we have the following result:

\begin{ble}
\label{lem1}
Let $G$ be a group generated by an abelian subgroup $A$ and an element $t\in
G$, such that $t^{-1}at=a^{-1}$ for any $a\in A$ and $t^2 \in A$. Then

\begin{enumerate}
\item For every $g\not \in A$ we have that $K(g)=\{a^2: a\in A\}$

\item If $g\not \in A$ then $gK(g)=C_g$.
\end{enumerate}
\end{ble}

\begin{proof} The considerations above show that (1) holds. So,
let $g\theta \in gK(g)$. Since $g\not \in A$ we have that conjugation by $g$
inverts the elements of $A$. By (1) we have that $\theta=\varphi^2$ for some 
$\varphi \in A$. Setting $t=g\varphi $ we see easily that $t^{-1}g\theta t=g$%
.\end{proof}

\noindent {\bf Remark : } Note that item (2) of the previous Lemma holds
whenever the elements of $K(g)$ are squares and are inverted by $g$.

Recently, it was proved that the Zassenhaus conjecture is true for the class of finite groups $H=KX$, where  and $X$ is a cyclic group which is normal in $G$ and $K$ is Abelian, \cite{hrtwck}. It would be interesting to know if the conjecture still holds for groups   $G=KX$, where $K$ is a normal Abelian subgroup of $G$ without $2$-Sylow subgroup and $X$ is a cyclic subgroup.

\begin{bth}
\label{teo1}

Let $G=\langle t,\ A:t^2\in A ,\ t^{-1}at=a^{-1}\ ,\ \forall a\in A\rangle $
where $A$ is an abelian normal subgroup of $G$.  Then $G$ is an $p$-UT-group.
\end{bth}

\begin{proof} Let $\alpha \in \U_1(\mathbb{Z} G)$ be a torsion
unit and $g\in G $ an element of the support of $\alpha$. Suppose first that 
$o(g)=\infty $. Note that $[G:A]=2$ and hence, by Proposition \ref{prop1}, $\tilde{%
\alpha}(g)=0$. Secondly, suppose that $g\not \in A $. By Lemma~\ref{lem1}, we have
that $gK(g)=C_g$. Notice that $gK(g)$ is central in the quotient group $%
G/K(g)$ and hence, by \cite[Prop. $4$ ]{Bovdi}, we have that $\tilde{\alpha}(g)=
\sum_{h\in gK(g)}\alpha (g)=0 \ or \ 1 $. 

Finally we consider a torsion element $g\in T(A)$; since
the support of $\alpha$ is finite and $t^2\in A$, we may suppose that $A$ is
finitely generated. In particular $A$ is a polycyclic group and hence, by
\cite[5.4.15]{Robinson}, we have that there exist $H\lhd A$, which is torsion free and
of finite index. Note that, since $A$ is abelian and conjugation by $t$
inverts the elements of $A$, $H$ will also be normal in $G$. Consider the
quotient group $\overline{G}=G/H$. The group $\overline{G}$  is metabelian  and thus , by a result of \cite{Dokuchaev-1}, has the $p$-UT-property. 
Let  $%
\overline g$ be the projection of $g$ in $\overline G$. Then it is easily seen that 
$C_{\overline g}$ is the projection of the subset $S=\{b\in A:b= x^{-1}axh,\
h\in H\ x\in G\}$. Note that we may write $S$ as a disjoint union $%
S=C_{g}\cup S_1$ where $S_1=\{b\in S: h\neq 1\}$ is a 
normal subset of $G$ whose  elements are all of infinite order. Writing $S_1$ as a disjoint union of conjugacy classes,  we conclude, by
Proposition~\ref{prop1}, that $\sum_{h\in S}\alpha (h)=\tilde{\alpha}(g)$. Consider the projection $\Psi :\mathbb{Z} G \longrightarrow \mathbb{Z}{\overline G}$ and let $\beta=\Psi (\alpha)$. Then, since $%
\overline G$ is a $p$-UT-group, we have that $\sum_{h\in S}\alpha (h)= 
\sum_{{\overline h}\in C_{\overline g}}\beta(\overline h)=\tilde{\beta}(\overline g)\in \{0,1\}$.
Hence   $%
\tilde{\alpha}(g)\in \{0,1\}$ for every element $g\in G$. Since $\alpha$ has augmentation 1, it follows that $G$ has the $p$-UT property. 
\end{proof}

A group $G$ is called a T-group if normality is transitive in $G$. Let $G$ be
a solvable T-group and set $A=C_{G}(G^{\prime})$. If $A$ is not a torsion
group then, by a result of \cite[13.4.9]{Robinson}, we have that $G$ satisfies the
condition of Theorem~\ref{teo1} and hence $G$ is a $p$-UT-group.

We now consider groups $G$ whose  derived subgroup  is cyclic of
infinite order, say $G^{\prime}=\langle \rho \rangle $. We shall use this
notation in the following results.

\begin{ble}
\label{4item}
Let $G$ be a group with cyclic derived subgroup; then

\begin{enumerate}
\item If $g\in T(G)$ centralizes $\rho$ then $g$ is central.

\item Elements of odd order are central.

\item $\{g^2:g\in T(G)\} \subseteq {\cal Z}(G)$.

\item If $g\in G$ has infinite order and $\alpha \in \U_1(\mathbb{Z} G)$
is an element of finite order then $\tilde{\alpha}(g)=0$.
\end{enumerate}
\end{ble}

\begin{proof}(1) Let $g\in T(G)$ and $x \in G$ then, since $%
\langle \rho \rangle $ is normal in $G$, we have that $g^{-1}xg=x \rho^k$
for some integer $k$. Let $m=o(g)$ then we have that $x=g^{-m}xg^{m}=x
\rho^{km}$. Since $\rho$ has infinite order we must have that $k=0$.

(2) If $g\in G $ then $g^2$ centralizes $\rho$ and hence is central. Since $g
$ has odd order we have that  $g$ is central.

(3) The proof of (2) applies.

(4) Suppose that this is false; then, by \cite[Prop. 2 ]{Bovdi}, there exist $k>1
,\ x\in G$ such that $x^{-1}gx=g^k$. This implies that $g^{k-1}=[g,x]\in
G^{\prime}$. Set $n=k-1\ and\ h=g^n$; then the subgroup $\langle h \rangle $
is normal in $G$. Hence $x^{-1}hx\in \{h,h^{-1}\}$. But on the other hand $%
x^{-1}hx=h^k$ and hence we must have that $k=1$, a contradiction.\end{proof}

\begin{ble}
\label{lem3}
Let $G$ be a group such that $G^{\prime}$ is infinite cyclic. Then, for any
torsion element $g\in G $, we have that $gK(g)=C_g$.
\end{ble}

\begin{proof} Let $G^{\prime}=\langle \rho \rangle$. Then, since $%
G^{\prime}$ is a normal subgroup, we have that $g^{-1}\rho g\in
\{\rho,\rho^{-1}\}$. If $g^{-1}\rho g=\rho$ then, by Lemma~\ref{4item}, $g$ is
central. So we may suppose that $g^{-1}\rho g=\rho^{-1}$. In this case also $%
g\rho g^{-1}=\rho^{-1}$. Hence we have that $g^{-1}g\rho^{-1}g=g\rho$, i.e.,  $%
g\rho^{-1}$is conjugated to $g\rho$. We now separate the proof in two cases: 

\noindent {\bf Case 1: } $K(g)\neq G^{\prime}$. \newline
Since $g^{-1}\rho g=\rho^{-1}$ and $K(g)$ is cyclic we must have that $%
K(g)=\langle \rho^2 \rangle$. Hence, by the Remark following Lemma~\ref{lem1}, we
have that $gK(g)=C_g$.

\noindent {\bf Case 2: } $K(g)= G^{\prime}$. \newline
In this case, since $G^{\prime}$ is cyclic and $\rho$ is inverted by
elements not in its centralizer, we see easily that there is an element $%
t\in G $ such that $K(g)=\langle [g,t] \rangle $. In particular, we have
that $[g,t]\in \{\rho,\rho^{-1}\}$. Hence $g$ is conjugated either to $g\rho
\ or \ to \ g\rho^{-1}$. Since we have already proved that $g\rho$ is
conjugate to $g\rho^{-1}$, we only have to prove that an element of $gK(g)$
is either conjugate to $g$ or to $g\rho$. In fact, set $h=g\theta$ with $%
\theta \in K(g)$. If $\theta$ were a square then, by the Remark following
Lemma~\ref{lem1}, $h$ is conjugate to $g$. If $\theta$ is not a square, we may write 
$h=g\rho \varphi$ where $\varphi$ is a square. Hence, again by the same
Remark, we have that $h$ is conjugated to $g\rho$ which in turn is
conjugated to $g$.\end{proof}

\begin{bth}
Let $G$ be a group such that the derived subgroup of $G$ is infinite cyclic.
Then $G$ is a UT-group.
\end{bth}

\begin{proof} Let $\alpha \in \U_1(\mathbb{Z} G)$ be a torsion
unit and $g\in G$ an element. If $g$ is of infinite order then, by Proposition~\ref{prop1},
we have that $\tilde{\alpha }(g)=0$. If $g$ is a torsion element then, by
Lemma~\ref{lem3}, we have that $\tilde{\alpha}(g)=\sum_{h\in gK(g)}\alpha (h)$. 
Since the element $gK(g)$ is central in the quotient group $G/K(g)$ we
have, by \cite[Prop. $4$]{Bovdi}, that $\sum_{h\in gK(g)}
\alpha(h)\in\{0,1\}$. Since $\alpha$ has augmentation 1, the result is
proved. \end{proof}


Let $G$ be a group and $(A_n)$ a descending chain of normal subgroups of $G$. 
Denote by $\Psi_n :G\longrightarrow G/A_n$ the natural map and let $F_n $
be the pre-image of $\Psi_n (C_{g_0})$, with  $g_0\in G$. In what follows we shall use this
notation.

\begin{bpr}
\label{conj}
Let $G$ be a group and $(A_n)$ a descending chain of normal subgroups of $G$
such $\bigcap A_n=1$. Then, with the notation above, for every element $g_0\in
G$   we have that $\bigcap F_n=C_{g_0}$. 
\end{bpr}

\begin{proof} Let $\Psi_n :G\longrightarrow G/A_n$ and $F_n $ be
as above. Clearly $F_n$ is a normal subset of $G$ so we may write it as a
disjoint union of conjugacy classes, say $F_n=\bigcup C_{h_{nj}}$. Note that
each $h_{n_j}$ is either in $C_{g_0}$ or is not in $C_{g_0}$ and is of the form $ 
h_{nj}=g_0\varphi_{nj}$ with $\varphi_{nj}\in A_n$. Since the family $(A_n)$
is descending, we have that $F_{n+1}\subset F_n$. Now suppose that an
element $h=g_0\varphi$ appears in $F_n$ and in $F_{n+1}$ as a representative
of a conjugacy class; then $\varphi\in A_{n+1}$. So if $h=g_0\varphi $
appears in every $F_n$ then it follows that $\varphi\in \bigcap A_n=1$.
Hence $\bigcap F_n=C_{g_0}$.\end{proof}

We still denote by $\Psi_n$ the extension of $\Psi_n :G\longrightarrow G/A_n$
to the group rings $\mathbb{Z} G$ and $\mathbb{Z} G/A_n$. If 
$\alpha\in {\cal U}_1 \mathbb{Z} G, g\in G$ then put $\beta_n=\Psi_n
(\alpha)$ and ${\overline g}=\Psi_n (g)$.

\begin{bth}
 Let $G$, $A_n$, $\alpha$ and $\beta_n$ be as above. Given an element $%
g_0\in G$ there exists $n_0\in \mathbb{N}$ depending on $g_0$ such that $\widetilde\beta_{n_0} ({%
\overline g_0})=\widetilde\alpha (g_0)$.
\end{bth}

\begin{proof} Since $\alpha$ has finite support we can choose a
finite number of elements of $G$, say, $g_1,\ldots,g_k$, representing the
elements of the support of $\alpha$. By Proposition~\ref{conj}, for every $1\leq j\leq k
$ there is an index $m_j$ so that $g_0$ and $g_j$ are not conjugate in $%
G/A_{m_j}$. Put $n_0=max\{m_j\}$; then $g_0$ is not conjugate to $g_j$ in $%
G/A_{n_0}$ for every $1\leq j \leq k$. It follows that $\widetilde\beta
(\overline{g_0})=\sum_{{\overline g}\sim {\overline g_0}}\alpha
(g)=\widetilde\alpha (g_0).$\end{proof}

\begin{bco}
Let $G$ and $(A_n)$ be as in the previous theorem. If each $G/A_n$ is a
UT-group then $G$ is also a UT-group.
\end{bco}

\begin{proof} Since each $G/A_n$ is a UT-group we have that that $%
\alpha (g_0)\in \{0,1\}$ and hence $G$ is a UT-group.\end{proof} 

\begin{bth}
Let $G$ be a polycyclic group and suppose that every finite quotient of $G$
is UT-group. Then $G$ is a UT-group.
\end{bth}

\begin{proof} We use induction on the Hirsch length of $G$. It is
clear that we may suppose that $G$ is not finite. By \cite{Robinson}, $G$
contains an abelian normal torsion free subgroup $A$. Setting $A_n=A^n$ we obtain a
descending chain. Since $A$ is an abelian polycyclic group we have that $%
\bigcap A_n=1$. Now every $G/A_n$ has shorter Hirsch length then $G$ and
since every finite quotient of $G/A_n$ is isomorphic to a finite quotient of 
$G$ it follows that each $G/A_n$ is a UT-group. The result follows by the
previous corollary.\end{proof}

Note that the theorem says that policyclic groups are UT-groups if and only if every
finite soluble group is a UT-group. So, the conjecture of Zassenhaus for
finite groups would imply that every polycyclic group is a UT-group.

We now look at  the p-UT property. Note that the former results also
apply in this case.

Let ${\cal F}$ be a family of finite groups and $G$ an arbitrary group. We
say that $G$ is an ${\cal F}$-group if every finite quotient of $G$ is in $%
{\cal F}$. It is easy to see that every quotient of an ${\cal F}$-group is
also an ${\cal F}$-group. Let we denote the set of finite soluble groups by $%
{\cal F}_s$ then it is clear that every polycyclic group is an ${\cal F}_s$%
-group. We denote also by ${\cal F}_f$, ${\cal F}_{ni}$, ${\cal F}_4$
respectively the families of finite Frobenius groups, groups with nilpotent
derived subgroup and solvable groups whose order is not divisible by $p^3$, where $p$ is any prime.

By results of \cite{Dokuchaev-1, Juriaans, Polcino} the families ${\cal F}_f,\ {\cal F}_{ni},\ {\cal F}
_4$ all have the $p$-UT-property and so we have the following result.

\begin{bth}
Let $G$ be a polycyclic group. If $G$ is an ${\cal F}$-group, with ${\cal F}$ being  one of the
families above, then $G$ is a $p$-UT group.
\end{bth}

\begin{bth}
Let $G$ be a polycyclic group with nilpotent derived subgroup. Then $G$ is a
p-UT group. In particular supersoluble groups are p-UT groups.
\end{bth}

\noindent {\bf Acknowledgment:} The first author is grateful to the Departamento de Matemática of the 
Universidade Federal of Paraíba, Brazil, for financial support and its warm hospitality.

\vspace{5mm}

\noindent Juriaans, S.O.,\newline
Instituto de Matemática e Estatística \newline
Universidade de São Paulo \newline
CP 66281 \newline
CEP 05311-970 \newline
São Paulo - Brazil\newline
ostanley@usp.br

\vspace{5mm}

\noindent De A. E Silva, A., \newline
Departamento de Matem\'atica \newline
Universidade Federal da Para\'iba \newline
CEP 58051-900, \newline
Jo\~ao Pessoa, Pb, Brazil, \newline
andrade@mat.ufpb.br

\vspace{5mm}

\noindent Souza Filho, A. C., \newline
Escola de Artes, Ciências e Humanidade \newline
Universidade de São Paulo \newline
Rua Arlindo Béttio, 1000 \newline
CEP 03828-000 \newline
São Paulo - Brazil \newline
acsouzafilho@usp.br

\end{document}